\documentclass[10pt]{amsart}
\usepackage[alphabetic,lite]{amsrefs}
\usepackage{amsmath}
\usepackage{amsthm}
\usepackage{mathrsfs}
\usepackage{amssymb}
\usepackage{array}
\usepackage[all]{xy}

\usepackage{enumitem}

\usepackage{xcolor}

\usepackage{hyperref}
\hypersetup{
    colorlinks=true,       
    linkcolor=blue,          
    citecolor=magenta,        
    filecolor=green,      
    urlcolor=purple           
}

    \newtheorem{Lem}{Lemma}[section]
    \newtheorem{Prop}[Lem]{Proposition}
    \newtheorem{Thm}[Lem]{Theorem}  
    \newtheorem{Cor}[Lem]{Corollary}

\theoremstyle{definition}

    \newtheorem{Def}[Lem]{Definition}
    \newtheorem{Exa}[Lem]{Example}
    \newtheorem{Rem}[Lem]{Remark}
    \newtheorem{Not}[Lem]{Notation}

\newcommand{\CM}{\mathcal{M}}
\renewcommand{\P}{\mathbb{P}}

\newcommand{\Z}{\mathbb{Z}}

\newcommand{\pg}{\pi_g}

\newcommand{\pgp}{\pg^+}
\newcommand{\pgpl}{$\pgp$-length}
\newcommand{\pgm}{\pg^-}
\newcommand{\pgml}{$\pgm$-length}

\newcommand{\pd}{\pi_d}

\newcommand{\pdl}{$\pd$-length}

\newcommand{\Pnd}{{\P^{N_d}_k}}
\newcommand{\Pndz}{{\P^{N_d}_\Z}}

\newcommand{\bp}{\P^2_k}
\newcommand{\bpz}{\P^2_\Z}
\newcommand{\bpdz}{{\bpz}^\vee}

\newcommand{\bpnd}{{\P^n_k}^\vee}
\newcommand{\bpcand}{{\P^{g-1}_k}^\vee}

\newcommand{\glz}{\mathrm{GL}_{3,\Z}}

\DeclareMathOperator{\ac}{H}
\DeclareMathOperator{\hess}{Hess}

\DeclareMathOperator{\spec}{Spec}
\DeclareMathOperator{\mino}{minors}

\DeclareMathOperator{\Jac}{Jac}
\DeclareMathOperator{\Pic}{Pic}

\newcommand{\aand}{\qquad{\textrm{and}}\qquad}

\newcommand{\due}[1]{\ddot{#1}}

\newcommand{\Mg}{{{\CM}_g}}
\newcommand{\Mgdue}{\due{\CM}_g}

\newcommand{\Mgu}{{{\CM}_{g,1}}}

\newcommand{\Sg}{{\mathcal{S}}_g}
\newcommand{\Sgdue}{\due{\mathcal{S}}_g}

\newcommand{\Sgp}{\Sg^+}
\newcommand{\Sgpdue}{\due{\mathcal{S}}_g^+}

\newcommand{\Sgm}{\Sg^-}
\newcommand{\Sgmdue}{\due{\mathcal{S}}_g^-}

\newcommand{\Scan}{\Sg^{\textrm{can}}}
\newcommand{\Scano}{\Sg^{\textrm{can},\circ}}
\newcommand{\Snc}{\Sg^{\textrm{nc}}}

\title[Complex, tropical, positive characteristic]{Complex and tropical counts via positive characteristic}

\author[Pacini]{Marco Pacini}
\address{Instituto de Matem\'atica, Universidade Federal Fluminense, 
Rio de Janeiro, Brazil}
\email{pacini.uff@gmail.com}

\author[Testa]{Damiano Testa}
\address{Mathematics Institute, University of Warwick,
Coventry, CV4 7AL, UK}
\email{adomani@gmail.com}



\begin{document}

\maketitle

\begin{abstract}
We reconcile the discrepancy between the complex and tropical counts of some enumerative problems reducing to positive characteristic.  Each problem that we consider suggests a prime with special behaviour.  Modulo this prime, the solutions coalesce in uniform clusters: at this special prime, the geometric and the tropical behaviours match.  As examples, we concentrate on inflection points of plane curves and theta-hyperplanes of canonical curves.
\end{abstract}

\section{Introduction}

This paper started with our desire to understand the difference between some enumerative problems over the complex numbers and their tropical analogues: we found that by suitably choosing the characteristic of the field, we can mimick the tropical behaviour geometrically.  For us, tropical geometry was just the inspiring motivation: our arguments use mainly techniques from algebraic geometry.  It would be interesting to see if the techniques that we develop here could be adapted to the tropical setting and yield a systematic definition of combinatorial multiplicity.

For a specific example, we illustrate the situation with the case of inflection points to plane cubics (Section~\ref{s:inf}).  Classically, a smooth plane cubic over the complex numbers admits exactly~$9$ inflection points.  The tropical analogue of the problem of inflection points of plane cubics yields~$3$ tropical inflection points.  While the tropical count does not match the complex count, it does match the number of inflection points of a general plane cubic over a field of characteristic~$3$.  Hence, the geometric intuition that better adapts to the tropical problem of inflection points comes from fields of characteristic~$3$.  More generally, the number of inflection points of a general plane curve of degree~$d$ is $3d(d-2)$ over a field of characteristic different from~$3$, while it is $d(d-2)$ both over fields of characteristic~$3$ and tropically (see~\cite{BL}).  In an unrelated generalization, the inflection points of plane cubics are connected to points of order~$3$ on elliptic curves.  Again, for a prime~$p$, the $p$-torsion points of a tropical abelian variety are modelled more closely by the $p$-torsion points on an abelian variety over a field of characteristic~$p$, than they are by the $p$-torsion points over a field of characteristic different from~$p$.  We will not pursue this generalization.

Similarly, the number of bitangent lines to a plane quartic over a field of characteristic different from~$2$ is~$28$ (Section~\ref{s:titc}).  Both over a field of characteristic~$2$ and tropically, a general plane quartic admits~$7$ bitangent lines.  For background on bitangents to plane quartics in characteristic~$2$, see \cites{der0,der,wall,narr}; for tropical bitangents, see \cites{bitr,llb,lm}.  Analogous results hold, more generally, for canonical curves of arbitrary genus and their theta-hyperplanes.  For background on theta-characteristics in characteristic two, see~\cite{sv}, for tropical theta-characteristics, see \cites{zh,jl}.

We focus on problems where the count over the complex numbers differs from the tropical count.  Our initial observation is that, in all cases that we consider, the tropical count matches the enumerative count {\emph{over a field of a specially chosen prime characteristic}}, rather than over the complex numbers.  Table~\ref{t:num} summarizes the examples that caught our attention.

\begin{table}[h]
    \centering
\begin{tabular}{|>{\centering\arraybackslash}m{135pt}|>{\centering\arraybackslash}m{55pt}|>{\centering\arraybackslash}m{45pt}|>{\centering\arraybackslash}m{38pt}|}
\hline
& Complex count & Tropical count & Special prime \\
\hline
Inflection points of a general plane curve of degree~$d$
& $3d(d-2)$ & $d(d-2)$ & $3$ \\[5pt]
\hline
Bitangents of a general plane quartic
&
$28$ & $7$ & $2$ \\[5pt]
\hline
Theta-hyperplanes of a general canonical curve of genus~$g$
&
$2^{g-1}(2^g-1)$ & $2^g-1$ & $2$ \\[5pt]
\hline
\end{tabular}
\vspace{5pt}
    \caption{Examples of enumerative problems}
    \label{t:num}
\end{table}

Lemmas~\ref{l:pl3} and~\ref{l:lu2} highlight the presence of a special characteristic modulo which there is a concentration of degree.  We view these lemmas as different manifestations of the same principle: the geometric count in positive characteristic takes into account the ramification contribution in the reduction from characteristic zero to positive characteristic.  Our local analysis shows that this can be justified by the fact that some morphism becomes inseparable in a special characteristic.

There is a significant overlap between the enumerative problems that we analyze and the ones that Harris considers in~\cite{har} over the complex numbers.  The argument in characteristic zero is based on the computation of the relevant monodromy groups.  It is interesting to note that we also rely on some monodromy information, via a result of Ekedahl, see Theorem~\ref{t:eke}.

A basic concept that plays an important role for us is the length of a zero-dimensional irreducible component of a fiber of a morphism.  Let $\pi \colon X \to Y$ be a morphism of schemes and let~$x$ be a geometric point of~$X$.  Assume that the irreducible component of the scheme-theoretic fiber $\pi^{-1}(\pi(x))$ of~$\pi$ containing~$x$ has dimension~$0$ and hence only has the point~$x$ in its support.  We isolate the following concept, since it plays a crucial role for us.

\begin{Def} \label{d:lun}
The {\emph{$\pi$-length of~$x$}} is the length of the irreducible component of the scheme-theoretic fiber $\pi^{-1}(\pi(x))$ containing~$x$.
\end{Def}

Depending on context, what we call $\pi$-length is also called {\emph{ramification index}}, {\emph{multiplicity}}, {\emph{weight}}, and so on.

We give here a heuristic description of where our argument hinges, again in the case of inflection points of plane curves.  Form the two incidence correspondences $(C,p)$ and $(C,\ell)$ consisting of a plane curve $C \subset \mathbb{P}^2$ and either an inflection point~$p$ or an inflection line~$\ell$.  The Gauss map is the morphism that assigns to a pair $(C,p)$ the pair $(C,\mathcal{T}_{C,p})$, where $\mathcal{T}_{C,p}$ is the tangent line (actually the inflection line) to~$C$ at~$p$.  For a fixed general curve~$C$ defined over an algebraically closed field, the Gauss map is a bijection of finite sets.  If the characteristic of the ground field is different from~$3$, then the Gauss map is not only bijective on points, but is, in fact, an isomorphism of finite schemes.  However, if the characteristic of the ground field is~$3$, then the Gauss map is inseparable of degree~$3$.  This justifies why the number of inflection points on a general plane curve of degree~$d$ in~$\mathbb{P}^2$ over a field of characteristic~$3$ is one third of the number over a field of characteristic different from~$3$.

The situation is analogous in the case of bitangents to plane quartics: the two points of tangency divide by~$2^2$ the cardinality of the set of bitangent lines, passing from fields of characteristic different from~$2$ to fields of characteristic~$2$.  In the case of theta-hyperplanes of a canonical curve of genus~$g$, the factor is $2^{g-1}$, since there are $g-1$ tangency conditions.

In the final section, we show how the method used in this paper, could potentially be used in several other situations.  This strategy could lead to proving that certain primes divide the numbers of solutions to enumerative problems over the complex numbers.

\subsubsection*{Acknowledgements}
The authors would like to thank Michel van Garrel, Diane Maclagan, Miles Reid, Israel Vainsencher for helpful discussions at various stages of this project.  In the process of writing this article, we used extensively the computer program Magma~\cite{magma}: this computational tool proved immensely helpful.

\section{Inflection points of plane curves}
\label{s:inf}

In this section, we study the inflection points of a plane curve, with special attention to the case of fields of characteristic~$3$.  Our starting point is an adaptation of Proposition~1.6 of~\cite{par}.

Let~$n$ be a non-negative integer, let~$R$ be a commutative ring with identity and set $R_n = R[x_0 , \ldots , x_n]$, the polynomial ring over~$R$ in $n+1$ variables.  Let $F \in R_n$ be a polynomial.  For $i \in \{ 0 , \ldots , n\}$, we recall the definition of the second Hasse derivative of~$F$ with respect to the $i$-th variable~$x_i$, that we denote by~$F_{\frac{x_ix_i}{2}}(x)$.  View~$F$ as a polynomial in the variable~$x_i$ with coefficients depending on the remaining variables and write
\[
F(x_0,\ldots,x_n) = \sum_j F_j(x_0,\ldots,\check{x_i}, \ldots , x_n) x_i^j.
\]
Define
\[
F_{\frac{x_ix_i}{2}}(x_0,\ldots,x_n) = \sum_j \binom{j}{2}F_j(x_0,\ldots,\check{x_i}, \ldots , x_n) x_i^{j-2}
.
\]
We observe that the identity $2F_{\frac{x_ix_i}{2}} = \frac{{\partial}^2 F}{{\partial} x_i^2}$ holds.  To simplify the notation, if $i,j \in \{0,\ldots,n\}$ are distinct indices, then we denote the usual second derivative $\frac{{\partial}^2 F}{{\partial} x_i x_j}$ by $F_{\frac{x_ix_j}{2}}$.  In our application, Hasse derivatives allow us to treat the case of fields of characteristic~$2$ uniformly.

Let $S= R_n[t_0 , \dots , t_n]$ be a polynomial ring over~$R_n$ in $n+1$ further variables and set
\[
x = (x_0 , \ldots , x_n) , \quad
t = (t_0 , \ldots , t_n)
\aand
x+t = (x_0 + t_0 , \ldots , x_n + t_n).
\]
Combining the second Hasse derivatives with Taylor expansion, we obtain the identity
\[
F(x+t)
=
F(x) + 
\sum _{i=0}^n F_{x_i}(x) t_i + 
\sum_{0 \leq i \leq j \leq n} F_{\frac{x_ix_j}{2}}(x) t_it_j + 
G,
\]
where $G \in S$ is a polynomial in the cube of the ideal generated by $t_0,\ldots,t_n$.  Set
\begin{eqnarray*}
\nabla_x F \cdot t
& = &
\sum _{i=0}^n F_{x_i}(x) t_i, \\
\hess_x F (t)
& = &
\sum_{0 \leq i \leq j \leq n} F_{\frac{x_ix_j}{2}}(x) t_it_j,
\end{eqnarray*}
so that the congruence
\[
F(x+t) \equiv F(x) + \nabla_x F \cdot t + \hess_x F (t)
\mod{(t_0,\ldots,t_n)^3}
\]
holds.

Denote by~$\bpdz$ the projective plane dual to~$\bpz$ and choose homogeneous coordinates $\ell_0,\ell_1,\ell_2$ on $\bpdz$ that are dual to the coordinates $x_0,x_1,x_2$ on~$\bpz$.  Thus, if $\ell \subset \bpz$ is a line, then we identify~$\ell$ with the point $[\ell_0,\ell_1,\ell_2] \in \bpdz$ and~$\ell$ is the line with equation
\begin{equation} \label{e:eql}
\ell \colon \quad
\ell_0 x_0 + \ell_1 x_1 + \ell_2 x_2 = 0.
\end{equation}
We denote by $\mathcal{F} \subset \bpz \times \bpdz$ the closed subscheme defined by Equation~\eqref{e:eql}.  The scheme~$\mathcal{F}$ consists of pairs $(p,\ell)$ consisting of a point~$p$ in the projective plane and a line~$\ell$ containing~$p$.  We call~$\mathcal{F}$ the {\emph{flag variety}}.

Associated to the line~$\ell$, we also find it convenient to introduce the set $\ell^\perp$ defined as
\[
\ell^\perp = \left\{
\begin{array}{l}
(0 , \ell_2 , - \ell_1) , \\[4pt]
(- \ell_2 , 0, \ell_0) , \\[4pt]
(\ell_1 , - \ell_0 , 0)
\end{array}
\right\}.
\]
The set~$\ell^\perp$ consists of vectors spanning the kernel of the linear equation~\eqref{e:eql} determined by~$\ell$.

We set $N_d = \binom{d+2}{2}-1$, the dimension of the projective space of plane curves of degree~$d$, that we identify with the non-zero forms of degree~$d$, up to scaling, that vanish on the corresponding curve.

\begin{Def}
The {\emph{universal inflection flag}} is the closed subscheme $\Gamma_d$ in $\Pndz \times \bpz \times \bpdz$ defined by
\[
\Gamma_d =
\left\{
(F,p,\ell) \in \Pndz \times \bpz \times \bpdz \quad \left| \quad
\begin{array}{l}
\displaystyle
p \in \ell, \quad F(p)=0 , {\textrm{ and,}}
\\[5pt]
\displaystyle
{\textrm{for all }} v \in \ell^\perp, \\[5pt]
\displaystyle
\nabla_p F \cdot v = 0 , \quad
\hess_p F (v) = 0 .
\end{array}
\right.
\right\}.
\]
We also define
\begin{itemize}
\item
the {\emph{universal inflection point}} as the image~$P_d$ of~$\Gamma_d$ under the projection $\pi_{12} \colon \Gamma_d \to \Pndz \times \bpz$, and
\item
the {\emph{universal inflection line}} as the image~$L_d$ of~$\Gamma_d$ under the projection $\pi_{13} \colon \Gamma_d \to \Pndz \times \bpdz$.
\end{itemize}
\end{Def}

The commutative diagram~\eqref{e:losa}, where all the arrows are standard projections, summarizes and extends the notation.

\begin{equation} \label{e:losa}
\xymatrix{
& \Gamma_d \ar[rr]^{\pi_{23}} \ar[dl]_{\pi_{12}} \ar[dr]^{\pi_{13}} \ar[dd]^{\pi_{1}} & & \mathcal{F} \\
P_d \ar[dr]_{\pd} & & L_d \ar[dl]^{\pi_f} \\
& \Pndz
}
\end{equation}

\begin{Rem}
We can use Diagram~\eqref{e:losa} to explain the difference between inflection points and inflection lines.  Let~$k$ be a field and denote by $(\Gamma_d)_k$, $(P_d)_k$ and $(L_d)_k$ the base-change to $\spec k$ of the schemes $\Gamma_d$, $P_d$ and $L_d$.  The morphism $(\Gamma_d)_k \to (P_d)_k$ is birational, regardless of the field~$k$.  The morphism $(\Gamma_d)_k \to (L_d)_k$ is birational if and only if the characteristic of~$k$ is different from~$3$.  If the characteristic of~$k$ is~$3$, then the scheme~$(L_d)_k$ is generically non-reduced and the morphism $(\Gamma_d)_k \to (L_d)_{k,{\mathrm{red}}}$ between reduced induced schemes is purely inseparable of degree~$3$.
\end{Rem}

\begin{Def}
Let~$d$ be a positive integer, let~$k$ be a field and let $F \in \Pnd$ correspond to the plane curve $C \subset \mathbb{P}^2_k$ of degree~$d$.  The {\emph{inflection flag scheme of~$C$}} is the pull-back~$\mathcal{F}_C \subset \mathcal{F}$ of~$\pi_1$ over the $k$-valued point~$F$.
\end{Def}

\begin{Exa} \label{ex:if3}
Let~$k$ be a field of characteristic~$3$ and let $E \subset \mathbb{P}^2_k$ be the plane cubic with equation
\[
F_E = x_0x_1x_2 + (x_0-x_1)^3 = 0.
\]
The equations of the inflection flag scheme~$\mathcal{F}_E$ of~$E$ are
\[
\mathcal{F}_E \colon \qquad
\left\{
\begin{array}{rcl}
\ell_0 x_0 + \ell_1 x_1 + \ell_2 x_2
& = & 0 , \\[5pt]
x_0x_1x_2 + (x_0-x_1)^3
& = & 0 , \\[5pt]
2 \times 2 {\textrm{ minors of }}
\begin{pmatrix}
\ell_0 & \ell_1 & \ell_2 \\
x_1x_2 & x_0x_2 & x_0x_1
\end{pmatrix}
& = & 0 , \\[9pt]
\ell_1 \ell_2 x_0
\; \; \; = \; \; \;
\ell_0 \ell_2 x_1
\; \; \; = \; \; \;
\ell_0 \ell_1 x_2
& = & 0 .
\end{array}
\right.
\]
An easy, direct calculation shows that~$\mathcal{F}_E$ consists of the three points
\begin{itemize}
\item
$([1,1,0],[0,0,1])$ with one-dimensional Zariski tangent space;
\item
$([0,0,1],[1,0,0])$ and $([0,0,1],[0,1,0])$, both with two-dimensional Zariski tangent space.
\end{itemize}
The multiplicity of each one of the three points in the support of~$\mathcal{F}_E$ is~$3$.  The point $[1,1,0]$ is smooth on~$E$ and the tangent line to~$E$ at such point has intersection multiplicity~$3$ with~$E$.
\end{Exa}

In the next lemma, we compute the $\pi_{13}$-length (Definition~\ref{d:lun}) of certain points of~$\Gamma_d$.

\begin{Lem} \label{l:pl3}
Let $(F,p,\ell) \in \Gamma_d$ be a $k$-valued point and let $C \subset \mathbb{P}^2_k$ be the curve with equation $F=0$.  Assume that~$C$ contains no line through~$p$ and that the intersection multiplicity of~$C$ and~$\ell$ at~$p$ is~$3$.  The $\pi_{13}$-length~$\lambda_{(F,p,\ell)}$ of $(F,p,\ell)$ is
\begin{itemize}
\item
$\lambda_{(F,p,\ell)} = 1$, if the characteristic of~$k$ is different from~$3$;
\item
$\lambda_{(F,p,\ell)} = 3$, if the characteristic of~$k$ is~$3$.
\end{itemize}
\end{Lem}

\begin{proof}
Choose homogeneous coordinates $x_0,x_1,x_2$ on~$\bp$ so that~$p$ is the point $[0,0,1]$ and~$\ell$ is the line with equation $\ell \colon x_0 = 0$.  Since the intersection multiplicity of~$C$ and the line $x_0=0$ is~$3$, there are forms $G(x_0,x_1,x_2) \in k[x_0,x_1,x_2]$ and $H(x_1,x_2) \in k[x_1,x_2]$ with $H(0,1) \neq 0$ such that we can write
\[
F(x_0,x_1,x_2) = x_0 G(x_0,x_1,x_2) + x_1^3 H(x_1,x_2).
\]
The equations of the scheme-theoretic fiber~$\Pi$ of~$\pi_{13}$ at the point $(F,\ell)$ are
\[
\Pi \colon \quad
\left\{
\begin{array}{rcccl}
0 & = & x_0, \\[5pt]
0 & = & x_0 G & + & x_1^3 H,
\\[5pt]

0 & = & x_0 G_{x_1} & + & x_1^{2} \left(
3 H + x_1 H_{x_1} \right) ,
\\[5pt]
0 & = & x_0 G_{x_2} & + & x_1^3 H_{x_2},
\\[5pt]
0 & = & x_0 G_{\frac{x_1x_1}{2}} & + & x_1 \left(
3 H + 3 x_1 H_{x_1} + x_1^2 H_{\frac{x_1x_1}{2}} \right) ,
\\[5pt]
0 & = & x_0 G_{\frac{x_2x_2}{2}} & + & x_1^3 H_{\frac{x_2x_2}{2}}.
\end{array}
\right.
\]
The first two equations imply the identity $x_1^3 H = 0$.  Since~$H$ does not vanish at~$p$, we deduce that~$x_1^3$ is in the ideal of the irreducible component~$(\Pi)_p$ of~$\Pi$ supported at~$p$.  Thus, the scheme~$(\Pi)_p$ is defined by the equations
\[
(\Pi)_p \colon \quad
x_0 = x_1^3 = 3 x_1 = 0.
\]
The result follows.
\end{proof}

\begin{Lem}\label{l:girri}
The projection map $\pi_{23} \colon \Gamma_d \longrightarrow \mathcal{F}$ defined by $(F,p,\ell) \mapsto (p,\ell)$ is smooth and its fibers are linear subspaces of~$\Pndz$ of codimension~$3$.
\end{Lem}

\begin{proof}
The equations defining~$\Gamma_d$ are linear in the coefficients of~$F$, showing that the scheme-theoretic fibers of~$\pi_{23}$ are linear subspaces of~$\Pndz$.  Thus, to prove the statement, it suffices to show that the codimension of each geometric fiber of~$\pi_{23}$ is~$3$.  Observe also that, by the standard covariance properties of the gradient and the Hessian, the scheme~$\Gamma_d$ is stable under the natural action of the group-scheme~$\glz$.

Let~$k$ be a field and let $(p,\ell)$ be a $k$-valued point of the flag variety~$\mathcal{F}$.  Choose coordinates on~$\bp$ so that~$p$ is the point $[0,0,1]$ and~$\ell$ is the line $\ell \colon x_0 = 0$.  Write $F = \sum a_{i_0i_1i_2} x_0^{i_0} x_1^{i_1} x_2^{i_2}$, with coefficients $a_{i_0i_1i_2}$ in~$k$.  In this coordinate system, the equations defining the fiber of~$\pi_{23}$ over $(p,\ell)$ reduce to
\[
a_{00d} = 0 , \qquad
a_{01d-1} = 0 , \qquad
a_{02d-2} = 0 , \qquad
\]
and the result follows.
\end{proof}

\begin{Lem} \label{l:gm}
The morphism $\pi_{12} \colon \Gamma_d \to P_d$ sending $(F,p,\ell)$ to $(F,p)$ is an isomorphism on the open set where~$p$ is a smooth point of the curve with equation $F=0$.
\end{Lem}

\begin{proof}
The assignment
\begin{eqnarray*}
P_d & \dashrightarrow & \Gamma_d \\
(F,p)
& \longmapsto &
\left(
F,p,[F_{x_0}(p),F_{x_1}(p),F_{x_2}(p)]
\right)
\end{eqnarray*}
is a rational inverse of~$\pi_{12}$ and it is defined whenever at least one of the partial derivative of~$F$ does not vanish at~$p$, that is, whenever~$p$ is a smooth point of the curve $F=0$, as required.
\end{proof}

The following result is the numerical analogue, in characteristic~$3$, of the tropical results of \cite{BL}*{Section~5, especially Theorem~5.6}.

\begin{Thm} \label{t:lu3}
Let~$k$ be a field of characteristic~$3$ and let $F \in \Pnd$ be a general form of degree~$d$.  For every $k$-valued point $(F,p)$ in~$P_d$, the {\pdl} of $(F,p)$ is~$3$.
\end{Thm}

\begin{proof}
Let~$U$ be the locus of triples $(F,p,\ell) \in \Gamma_d$ satisfying the conditions
\begin{itemize}
\item
$p$ is a smooth point of the curve $C$ with equation $F=0$ and~$\ell$ is the tangent line to~$C$ at~$p$;
\item
the intersection multiplicity of the tangent line~$\ell$ with~$C$ at~$p$ is~$3$;
\item
the inflection flag scheme~$\mathcal{F}_{C}$ has dimension~$0$ and the point $(F,p,\ell)$ in the support of~$\mathcal{F}_{C}$ has length~$3$.
\end{itemize}
By Lemma~\ref{l:pl3} and standard upper-semicontinuity arguments, we see that~$U$ is open in~$\Gamma_d$.  Using again Lemma~\ref{l:pl3}, we obtain that for every point $u \in U$ the $\pi_{13}$-length of~$u$ is~$3$.  Applying Lemma~\ref{l:gm}, we deduce that for a triple $(F,p,C) \in U$, the $\pi_{13}$-length coincides with the {\pdl} of $(F,p)$. and they are therefore both equal to~$3$.

Let~$F_E$ be the form appearing in Example~\ref{ex:if3} and let $p_E \in E$ be the point $p_E = [1,1,0]$ and $\ell_E \subset \bp$ be the tangent line to~$E$ at~$p_E$.  It is easy to verify that the triple $(x_0^{d-3}F_E , p_E , \ell_E)$ is contained in~$U$, showing that~$U$ is non-empty.

Set $Z = \pi_1 (\Gamma_d \setminus U) \subset \Pnd$ and let $V \subset \Pnd$ be the complement of~$Z$.  By Lemma~\ref{l:girri}, the scheme~$\Gamma_d$ is irreducible and hence the open set~$V$ in~$\Pnd$ is dense.  It follows that every pair $(F,p)$, with~$F$ belonging to~$V$ has {\pdl} equal to~$3$, as required.
\end{proof}

Let~$d$ be a positive integer, let~$k$ be a field and let $C \subset \bp$ be the plane curve of degree~$d$ with equation $F=0$.  We denote by~$P_C$ the scheme-theoretic fiber $\pd^{-1}(F)$ of~$\pd$ over the $k$-valued point of~$\Pnd$ corresponding to~$C$.  We call~$P_C$ the {\emph{scheme of inflection points of~$C$}}.

\begin{Cor}
Let~$d$ be a positive integer, let~$k$ be an algebraically closed field of characteristic~$3$ and let $C \subset \bp$ be a general plane curve of degree~$d$.  The scheme of inflection points of~$C$ consists of $d(d-2)$ points, each of multiplicity~$3$.
\end{Cor}

\begin{proof}
The result follows from Theorem~\ref{t:lu3}, combined with the classical fact that, since $C$ is general, the scheme of inflection points $P_C$ of $C$ is finite of degree $3d(d-2)$.
\end{proof}

\section{Theta hyperplanes and theta characteristics}
\label{s:titc}

\subsection{Reminders on theta-characteristics}

Let~$k$ be a field and let~$C$ be a smooth, projective, geometrically irreducible curve of genus~$g$ defined over~$k$.  Denote by~$\omega_C$ the canonical divisor of~$C$.  A {\emph{theta-characteristic of~$C$}} is a line bundle~$\mathscr{L}$ on~$C$ such that there is an isomorphism $\mathscr{L}^2 \cong \omega_C$.

Assume that the characteristic of the field~$k$ is different from~$2$.  The curve~$C$ admits $2^{2g}$ distinct geometric theta-characteristics.  The {\emph{parity}} of the number of global sections of a theta-characteristic is deformation invariant.  This induces a partition on the set of theta-characteristics: a theta-characteristic~$\mathscr{L}$ is {\emph{even}} if the dimension $\dim \ac^{0} \left( C , \mathscr{L} \right)$ is even, it is {\emph{odd}} otherwise.  There are exactly $2^{g-1}(2^g+1)$ distinct even theta-characteristics and $2^{g-1}(2^g-1)$ distinct odd theta-characteristics.  If~$C$ is general, then every even theta-characteristic has no non-zero global sections and every odd theta-characteristic has a one-dimensional vector space of global sections.  We refer to the paper~\cite{Mu} of Mumford and to~\cite{ACGH}*{Appendix~B, p.\ 281}.

Assume now that the characteristic of the field~$k$ is~$2$.  There is an integer $h_C \geq 0$ such that the curve~$C$ admits $2^{g-h_C}$ distinct geometric theta-characteristics.  The curves for which~$h_C$ vanishes form a non-empty open subset of the moduli space of curves of genus~$g$: these curves are called {\emph{ordinary}} and will be our main focus of interest.  The parity of the number of global sections of a theta-characteristic is {\emph{no longer}} deformation invariant.  Rather than the parity, we will be concerned with whether or not a theta-characteristic admits a non-zero global section: we say that theta-characteristic is {\emph{effective}} if it admits a non-zero global section, we say that it is {\emph{non-effective}} otherwise.  The ordinary curves~$C$ admit a unique non-effective theta-characteristic: this theta-characteristic is called the {\emph{canonical theta-characteristic}}.  The notion of canonical theta-characteristic extends to all smooth projective curves.  An alternative definition of ordinary is a curve for which the canonical theta-characteristic is not effective.  On every curve all {\emph{non-canonical}} theta-characteristics are effective.  In particular, for a non-ordinary curve, all theta-characteristics are effective.  The theory of theta-characteristics in characteristic~$2$ is closely related to de Rham cohomology, the Cartier operator, Dieudonn\'e modules and so on.  We will not need any of this machinery.  We refer the interested reader to \cite{sv}*{Section~3} and to \cites{serre,bak:cartier,ah}.

\subsection{Local computations}

Let~$k$ be a field, let~$n$ be a positive integer and let $C \subset \mathbb{P}^n_k$ be a smooth, quasi-projective, curve.  For a $k$-valued point $p \in C$, we denote by $\mathcal{T}_{C,p}$ the embedded Zariski tangent space to~$C$ at~$p$.  Set
\[
C^\vee =
\left\{
(p,H) \mid p \in C {\textrm{ and }} H \supset \mathcal{T}_{C,p}
\right\}
\subset \mathbb{P}^n_k \times \bpnd ,
\]
so that~$C^\vee$ is a closed subset of $C \times \bpnd$.  Denote by $\gamma \colon C^\vee \to \bpnd$ the projection to the second factor.

In the next lemma, we compute the $\gamma$-length (Definition~\ref{d:lun}) of certain points of~$C^\vee$.

\begin{Lem} \label{l:lu2}
Let $(p,H) \in C^\vee$ be a $k$-valued point and assume that~$C$ contains no line through~$p$.  Assume that the intersection multiplicity of~$C$ and~$H$ at~$p$ is~$2$.  The $\gamma$-length~$\lambda_{(p,H)}$ of $(p,H)$ is
\begin{itemize}
\item
$\lambda_{(p,H)} = 1$, if the characteristic of~$k$ is different from~$2$;
\item
$\lambda_{(p,H)} = 2$, if the characteristic of~$k$ is~$2$.
\end{itemize}
\end{Lem}

\begin{proof}
Let $\mathscr{O}_{C,p}$ be the local ring of~$C$ at the point~$p$, with maximal ideal~$\mathfrak{m}_p$.  Choose a linear form~$x_0$ on~$\mathbb{P}^n_k$ such that the hyperplane with equation $x_0=0$ does not contain the point~$p$ and let $\mathbb{A}^n_k \subset \mathbb{P}^n_k$ be an affine coordinate chart defined by $x_0 \neq 0$, containing the point~$p$.

Let $V_p$ be the vector space of linear functions on~$\mathbb{A}^n_k$ vanishing on~$p$.  For each integer $r \geq 1$, let $V_p^r \subset V_p$ be the kernel of the $k$-linear map $V_p \to \mathscr{O}_{C,p}/\mathfrak{m}_p^r$.  Each successive quotient of the chain $V_p = V_p^1 \supset V_p^2 \supset \cdots \supset V_p^r \supset \cdots$ has dimension at most~$1$.  The elements in $V_p^1 \setminus V_p^2$ are linear functions vanishing on hyperplanes transverse to~$C$ at~$p$.  As~$C$ is smooth at~$p$, there are such hyperplanes: denote by~$x_1$ any linear form on~$\mathbb{P}^n_k$ such that the ratio $\frac{x_1}{x_0}$ is in $V_p^1 \setminus V_p^2$.  Let~$x_2$ be a non-zero linear form on~$\mathbb{P}^n_k$ vanishing on~$H$.  By the assumption that~$C$ and~$H$ have intersection multiplicity~$2$ at~$p$, we deduce that~$\frac{x_2}{x_0}$ is contained in $V_p^2 \setminus V_p^3$.  Finally, let $x_3, \ldots , x_n$ be linear forms on~$\mathbb{P}^n_k$ such that $\frac{x_3}{x_0}, \ldots , \frac{x_n}{x_0}$ is a basis of~$V_p^3$.  Observe that $x_0 , \ldots , x_n$ is a homogeneous coordinate system for~$\mathbb{P}^n_k$.  With respect to this basis,
\begin{itemize}
\item
the homogeneous coordinates of the point~$p$ are $[1 , 0 , \ldots , 0]$,
\item
the equation of the hyperplane~$H$ is $x_2 = 0$,
\item
the equations of the tangent line~$\mathcal{T}_{C,p}$ to~$C$ at~$p$ are $x_2 = \cdots = x_n = 0$.
\end{itemize}
Since the equations of the tangent line to~$C$ at~$p$ are $x_2 = \cdots = x_n = 0$, we can choose homogeneous polynomials $f_2,\ldots ,f_n \in k[x_0 , \ldots , x_n]$, vanishing on~$C$ and such that, for $i \in \{2 , \ldots , n\}$, there is a homogeneous polynomial $g_i(x_0,\ldots,x_n)$ in the square of the ideal generated by $(x_1,\ldots,x_n)$ satisfying the equality
\[
f_i (x_0 , \ldots , x_n) = x_0^{\deg f_i-1} x_i + g_i(x_0,\ldots,x_n).
\]
It follows that the scheme with equations $f_2 = \cdots = f_n = 0$ coincides with~$C$ in a neighbourhood of the point~$p$.

We now argue that the coefficient of the monomial~$x_0^{\deg f_2-2} x_1^2$ in~$g_2$ is non-zero.  We set $y_1 = \frac{x_1}{x_0} , \ldots , y_n = \frac{x_n}{x_0}$.  By construction, the restriction of all monomials in $y_1 , \ldots , y_n$ to~$C$ belongs to $\mathfrak{m}_p^3$, except for the monomials $1,y_1,y_2,y_1^2$.  Since the restriction of the polynomial $y_2+g_2(1,y_1 , \ldots , y_n)$ to~$C$ vanishes, we deduce that in the polynomial $g_2(1,y_1 , \ldots , y_n)$ the monomial $y_1^2$ appears with non-zero coefficient, as stated.

The equations
\begin{eqnarray}
\nonumber
\sum_{i=0}^n \ell_i x_i
& = &
0
\\
\label{e:cduale}
f_2 (x_0 , \ldots , x_n) = \cdots = f_n (x_0 , \ldots , x_n)
& = &
0 \\
\nonumber
n \times n \; \mino \left(
\Jac \left(
\sum \ell_i x_i , f_2 , \ldots , f_n \right)
\right)
& = &
0,
\end{eqnarray}
where $\Jac$ denotes the $n \times (n+1)$ Jacobian matrix, with respect to the variables $x_0 , \ldots , x_n$, define~$C^\vee$ in a neighbourhood of $(p,H)$ in $\mathbb{P}^n_k \times {\mathbb{P}^n_k}^\vee$.

Let $C^\vee_H$ be the fiber of~$\gamma$ over the hyperplane~$H$, and let $(C^\vee_H)_p$ be the irreducible component supported at $(p,H)$ of $C^\vee_H$.  We obtain the equations of $(C^\vee_H)_p$ in a neighbourhood of $(p,H)$, setting the $(n+1)$-tuple $(\ell_0 , \ldots , \ell_n)$ to $(0,0,1,0,\ldots , 0)$ in Equations~\eqref{e:cduale}.  Thus, scheme $(C^\vee_H)_p$ is contained in~$C$ and, using the equation $x_2=0$, we find that it is even contained in $\spec \mathscr{O}_{C,p}/\mathfrak{m}_p^2$.  We deduce that the scheme $(C^\vee_H)_p$ is reduced (and smooth) if and only if at least one of the $n \times n$ minors of $J=\Jac (x_2 , f_2 , \ldots , f_n)$ is not contained in~$\mathfrak{m}_p^2$.  The matrix~$J$ belongs to
\[
\begin{pmatrix}
0 & 0 & 1 & 0 & \cdots & 0 \\
\mathfrak{m}_p^2 & 2\lambda x_1 + \mathfrak{m}_p^2 & 1+\mathfrak{m}_p & \mathfrak{m}_p & \cdots & \mathfrak{m}_p \\
\mathfrak{m}_p^2 & \mathfrak{m}_p & \mathfrak{m}_p & 1+\mathfrak{m}_p & \ddots & \vdots \\
\vdots & \vdots & \vdots & \ddots & \ddots & \mathfrak{m}_p \\
\mathfrak{m}_p^2 & \mathfrak{m}_p & \mathfrak{m}_p & \cdots & \mathfrak{m}_p & 1+\mathfrak{m}_p
\end{pmatrix},
\]
where $\lambda$ is the coefficient of the monomial $x_0^{\deg f_2-2}x_1^2$ in the form $g_2$.  Since the entries of the first column of~$J$ are contained in~$\mathfrak{m}_p^2$, the only minor that may not lie in~$\mathfrak{m}_p^2$ is the one obtained removing the first column of~$J$: we denote this minor by~$M$.  Developing~$M$ with respect to the first row, we find that it is contained in
\[
\det
\begin{pmatrix}
2\lambda x_1 + \mathfrak{m}_p^2 & \mathfrak{m}_p & \cdots & \mathfrak{m}_p \\
\mathfrak{m}_p & 1+\mathfrak{m}_p & \ddots & \vdots \\
\vdots & \ddots & \ddots & \mathfrak{m}_p \\
\mathfrak{m}_p & \cdots & \mathfrak{m}_p & 1+\mathfrak{m}_p
\end{pmatrix}.
\]
Expanding the determinant using signed sums over permutations, we obtain the congruence $M \equiv 2\lambda x_1 \pmod{\mathfrak{m}_p^2}$.  We proved that~$\lambda$ is non-zero, and hence we find that $M \in \mathfrak{m}_p \setminus \mathfrak{m}_p^2$ if and only if the characteristic of~$k$ is different from~$2$, as needed.
\end{proof}

Let~$t$ be a positive integer and let $C_{2,t} \subset C^t \times \bpnd$ be the fiber product of~$C^\vee$ with itself~$t$ times over~$\bpnd$:
\[
C_{2,t} = {\underbrace{C^\vee \times_\gamma \cdots \times_\gamma C}_{t {\textrm{ times}}}}^\vee.
\]
We denote by $\gamma^\vee \colon C_{2,t} \to \bpnd$ the natural projection map.  Thus, the geometric points of $C_{2,t}$ consist of the $(t+1)$-tuples $(p_1,\ldots,p_t,H)$ such that the hyperplane~$H$ contains the tangent line to~$C$ at each one of the points $p_1 , \ldots , p_t$.

\begin{Prop}\label{prop:tp}
Let~$k$ be a field of characteristic~$2$ and let $n,t$ be positive integers.  Suppose that $C \subset \mathbb{P}^n_k$ is a smooth curve.  Let $\mathbf{p} = (p_1,\ldots,p_t,H)$ be a point in $C_{2,t}$.  Assume that the intersection multiplicity of~$H$ with each one of the points $p_1,\ldots,p_t$ is equal to~$2$.  The $\gamma^\vee$-length of $\mathbf{p}$ is~$2^t$.
\end{Prop}

\begin{proof}
Fix $i \in \{1,\ldots,t\}$.  Denote by $\gamma_i \colon C^t \times \bpnd \to C \times \bpnd$ the projection sending $(q_1,\ldots,q_t,H)$ to $(q_i,H)$ and by $F_i$ the irreducible component containing $(p_i,H)$ of the fiber of~$\gamma \colon C^\vee \to \bpnd$ over the point~$H$.  The support of the scheme~$F_i$ is the single point $(p_i,H)$ and the length of~$F_i$ is~$2$, by Lemma~\ref{l:lu2}.

Let~$F_{\mathbf{p}}$ denote the irreducible component containing~$\mathbf{p}$ of the fiber of~$\gamma^\vee$ over the point~$H$.  The support of the scheme~$F_{\mathbf{p}}$ is the single point~${\mathbf{p}}$.  The length of~$F_{\mathbf{p}}$ is the $\gamma^\vee$-length of~${\mathbf{p}}$.  The scheme~$F_{\mathbf{p}}$ satisfies
\[
F_{\mathbf{p}} = 
\gamma_1^{-1}(F_1) \cap \cdots \cap \gamma_t^{-1}(F_t) \cong
F_1 \times \cdots \times F_t.
\]
It follows that the $\gamma^\vee$-length of~${\mathbf{p}}$ is~$2^t$, as required.
\end{proof}

\subsection{Characteristic~2}

In this section, we work mostly with spaces defined over fields of characteristic~$2$.  Let $\mathcal{A} \to \spec{\mathbb{Z}}$ be a scheme (or stack) flat over $\spec{\mathbb{Z}}$.

\begin{Not} \label{n:pp}
We denote by $\due{\mathcal{A}}$ the fiber of $\mathcal{A} \to \spec{\mathbb{Z}}$ over the point $\spec{\mathbb{F}_2}$.
\end{Not}

Let $g \geq 0$ be a non-negative integer.  Let~$\Mg$ denote the stack of smooth projective connected curves of genus~$g$ flat over $\spec \Z$ and let 
$u \colon \Mgu \to \Mg$ be the universal curve over~$\Mg$.

Let $\Pic(\Mgu/\Mg) \to \Mg$ be the relative Picard stack.  We denote by
\begin{equation} \label{e:squaring}
\begin{array}{rcl}
\Pic(\Mgu/\Mg) & \longrightarrow & \Pic(\Mgu/\Mg) \\[5pt]
[(C,\mathscr{L})] & \longmapsto & [(C,\mathscr{L}^{\otimes 2})]
\end{array}
\end{equation}
the squaring morphism taking a line bundle~$\mathscr{L}$ on the curve~$C$ to the line bundle~$\mathscr{L}^{\otimes 2}$ on~$C$.  Let $\Sg \subset \Pic(\Mgu/\Mg)$ be the {\emph{stack of theta-characteristics}} defined as the fiber of the squaring morphism~\eqref{e:squaring} over the class of the relative canonical line bundle.  The stack~$\Sg$ has a natural morphism
\[
\pg \colon \Sg \longrightarrow \Mg.
\]
The geometric points of the fiber of~$\pg$ above a smooth projective curve~$C$ of genus~$g$ are the theta-characteristics of~$C$.

The stack~$\Sg$ is the union of two irreducible components, denoted by~$\Sg^+$ and~$\Sg^-$.  Each one of these components is finite and flat over~$\Mg$.  Over the open subset $\Mg \times_{\spec{\mathbb{Z}}} \spec{\mathbb{Z}[\frac{1}{2}]}$ of~$\Mg$, the stacks~$\Sg^+$ and~$\Sg^-$ are disjoint, \'etale, and their geometric points correspond to even or odd theta-characteristics.

For every prime~$p$ different from~$2$, the reductions $\Sg^\pm \times_{\spec{\mathbb{Z}}} \spec{\mathbb{F}_p}$ are irreducible (see~\cite{J}*{Section~3, Theorem~3.3.11}).  The stack~$\Sgdue$, reduction of~$\Sg$ modulo~$2$, is also reducible: the distinction between canonical and non-canonical theta-characteristics in characteristic~$2$ replaces the distinction between even and odd theta-characteristics in characteristic different from~$2$.

\begin{Def}
The stack $\Scan \subset \Sgdue$ of {\emph{canonical theta-characteristics}} is the closure of the (non-empty) open subset consisting of pairs $(C,\mathscr{L})$, where $\ac^0 \left( C , \mathscr{L} \right) = 0$.  The stack $\Snc \subset \Sgdue$ of {\emph{non-canonical theta-characteristics}} is the closure of the complement of~$\Scan$.
\end{Def}

\begin{Rem}
The subset~$\Scano$ consisting of pairs $(C,\mathscr{L})$, where $\ac^0 \left( C , \mathscr{L} \right) = 0$ is open in~$\Scan$ and hence also in~$\Sgdue$.  The image of $\Scano \to \Mgdue$ is the open subset of ordinary curves.  The natural closed immersion $\Snc \hookrightarrow \left( \Sgdue \setminus \Scano \right)$ is bijective on points.  Nevertheless, the morphism is an isomorphism only over the locus of ordinary curves.  Most of our arguments only involve ordinary curves, so that we do not need to worry too much about this distinction.
\end{Rem}

\begin{Thm}[Ekedahl] \label{t:eke}
The stacks~$\Scan$ and~$\Snc$ are irreducible.
\end{Thm}

\begin{proof}
This is a consequence of a more general theorem of Ekedahl (see~\cite{eke}*{p.~43, Theorem~2.1}).  The cited theorem computes explicitly the monodromy of the reduction modulo~$2$ of the $2$-torsion subgroup of $\Pic(\Mgu/\Mg) \to \Mg$.  The monodromy action is the natural action of $\mathop{GL}_g(\mathbb{F}_2)$ on~$(\mathbb{F}_2)^g$, which has two orbits.  Since the stack~$\Sgdue$ is a torsor under the action of the $2$-torsion subgroup, we deduce that~$\Sgdue$ has two irreducible components and we are done.
\end{proof}

As a consequence of the previous theorem, we see that
\begin{itemize}
\item
the support of the stack~$\Sgpdue$ is the support of the whole~$\Sgdue$,
\item
the support of the stack~$\Sgmdue$ is the same as the support of the stack~$\Snc$.
\end{itemize}

The stacks $\Sgp$ and $\Sgm$ meet over $\Mgdue$, the reduction of~$\Mg$ modulo~$2$ (see Notation~\ref{n:pp}).  This confirms that, over a field of characteristic~$2$, the parity of a theta-characteristic is no more an invariant under a deformation.

\begin{Exa}\label{ex:iper}
Let $g \geq 1$ be an integer and let~$k$ be field of characteristic~$2$.  Let~$X$ be a smooth, projective curve of genus~$g$ defined over~$k$ and let $h \colon X \to \P^1_k$ be a separable double cover.  Thus, if $g \geq 2$, then the curve~$X$ is hyperelliptic.  We assume that the curve~$X$ is ordinary, so that the support of the ramification divisor of~$h$ consists of~$g+1$ distinct geometric points $d_1 , \ldots , d_{g+1}$.  Assume that each one of the $g+1$ points $d_1 , \ldots , d_{g+1}$ is defined over~$k$.
\begin{itemize}
\item
The divisor $D = d_1 + \cdots + d_{g-1}$ is smooth and the line bundle $\mathscr{O}_X(D)$ is a theta-characteristic on~$X$.  By the geometric Riemann-Roch, the linear system~$|D|$ consists of the single effective, reduced divisor~$D$.
\item
The divisor $E = d_1 + \cdots + d_g - d_{g+1}$ is a non-effective theta-characteristic on~$X$.  Hence,~$E$ is the canonical theta-characteristic on~$X$ (see~\cite{sv}*{p.~60}).
\end{itemize}
\end{Exa}

In order to prove our results, we use certain curves that we describe in the following lemma.

\begin{Lem} \label{l:generale}
Let~$g$ be a non-negative integer.  There exists a smooth projective connected curve~$C$ of genus~$g$ defined over an algebraically closed field of characteristic~$2$ with the following property.  For every theta-characteristic~$\theta$ on~$C$ either
\begin{enumerate}
\item \label{i:noncano}
$\theta$ is not the canonical theta-characteristic of~$C$, in which case the dimension of the vector space $\ac^0 (C,\theta)$ is one and the unique effective divisor in the linear system~$|\theta|$ is reduced (that is, consists of exactly $g-1$ distinct points of~$C$); or
\item
$\theta$ is the canonical theta-characteristic of~$C$, in which case the dimension of the vector space $\ac^0 (C,\theta)$ is zero.
\end{enumerate}
\end{Lem}

\begin{proof}
The case of genus $g=0$ is clear: the curve~$C$ is isomorphic to~$\mathbb{P}^1$, the only theta-characteristic on~$\mathbb{P}^1$ is the canonical theta-characteristic $\mathscr{O}_{\mathbb{P}^1}(-1)$ which is not effective.

Suppose that the inequality $g \geq 1$ holds.  By Theorem~\ref{t:eke}, the stacks $\Sgdue$ are~$\Scan$ and~$\Snc$ are irreducible.  Thus, to prove the lemma, it is sufficient to exhibit a single curve with non-effective canonical theta-characteristic and an effective theta-characteristic satisfying item~\eqref{i:noncano}.  We conclude using Example~\ref{ex:iper}.
\end{proof}

In the next results, we deal with  {\emph{canonical curves}}.

\begin{Def}
A {\emph{canonical curve}} over a field~$k$ is a smooth, projective, geometrically irreducible curve $C \subset \mathbb{P}^n_k$ such that
\begin{itemize}
\item
the restriction~$\mathscr{O}_{C}(1)$ of the line bundle $\mathscr{O}_{\mathbb{P}^n_k}(1)$ to~$C$ is the canonical line bundle on~$C$, and
\item
the restriction on global sections
\[
\ac^0 \left(\mathbb{P}^n_k , \mathscr{O}_{\mathbb{P}^n_k}(1) \right)
\longrightarrow
\ac^0 \left(C , \mathscr{O}_{C}(1) \right)
\]
is an isomorphism.
\end{itemize}
\end{Def}

Let $C \subset \mathbb{P}^n_k$ be a canonical curve.  It follows from the definition that the curve~$C$ is not hyperelliptic, the genus~$g$ of~$C$ satisfies inequality $g \geq 3$, the dimension of the ambient projective space is $n=g-1$ and the degree of~$C$ is $2g-2$.

For any fixed integer $g \geq 3$, the space of canonical curves of genus~$g$ is birational to a $\mathbb{P}GL_{g}$-bundle over the moduli space of curves of genus~$g$.  In particular, this space is irreducible.

Caporaso and Sernesi introduce the following terminology in~\cite{csodd}*{Section~2.1}.

\begin{Def} \label{d:tc}
Let $C \subset \mathbb{P}^{g-1}_k$ be a canonical curve of genus~$g$.  A {\emph{theta-hyperplane}} $H \subset \mathbb{P}^{g-1}_k$ is a hyperplane such that the intersection multiplicities of $C \cap H$ are all even.
\end{Def}

Let~$C$ be a general canonical curve of genus $g \geq 3$ over a field $k$.  The curve~$C$ has a finite number of theta-hyperplanes.
\begin{itemize}
\item
If~$k$ is a field of characteristic different from~$2$, then the number of theta-hyperplanes of~$C$ is $2^{g-1}(2^g-1)$.
\item
If~$k$ is a field of characteristic~$2$, then the number of theta-hyperplanes of~$C$ is~$2^g-1$.  Indeed, by Lemma~\ref{l:generale} the theta-hyperplanes of a general curve are in bijection with the non-canonical theta-characteristics of~$C$ and an ordinary curve has~$2^g-1$ non-canonical theta-characteristics.
\end{itemize}

We set $C_{2,g-1}^\circ \subset C_{2,g-1}$ to be the open subset consisting of the complement of all diagonals, that is, $C_{2,g-1}^\circ$ is the open subset of $C_{2,g-1}$ whose geometric points correspond to the $g$-tuples $(p_1,\ldots,p_{g-1},H) \in C_{2,g-1}$ such that the points $p_1 , \ldots , p_{g-1}$ are {\emph{distinct}} and the hyperplane~$H$ contains the tangent line to~$C$ at each one of these points.  Thus, the image under the morphism $\gamma^\vee \colon C_{2,g-1} \to \bpcand$ of the open subset $C_{2,g-1}^\circ$ is contained in the set of all theta-hyperplanes of~$C$.

For a general canonical curve~$C$, the image $\gamma^\vee (C_{2,g-1}^\circ)$ is the set of all $2^g-1$ theta-hyperplanes of~$C$.  This allows us to define the {\emph{scheme of theta-hyperplanes}} of the general curve~$C$ as the scheme-theoretic image~$\Theta(C)$ of~$C_{2,g-1}^\circ$ under the morphism~$\gamma^\vee$.

In the genus $g=3$ case, the tropical version of the following corollary was conjectured in \cite{bitr}*{Conjecture~1.1} and proved in \cite{cj}*{Theorem~1.1} and \cite{lm}*{Theorem~4.1}.  In the genus $g=4$ case, again in the tropical setting, see also \cite{hl}*{Theorem~5.2}.

\begin{Thm} \label{thm:rie}
Let $g \geq 3$ be an integer and let~$k$ be a field of characteristic~$2$.  Let $C \subset \mathbb{P}^{g-1}_k$ be a general canonical curve.  The scheme of theta-hyperplanes of~$C$ is finite of degree $2^{g-1}(2^g-1)$ and the multiplicity of each is~$2^{g-1}$.
\end{Thm}

\begin{proof}
Since~$C$ is general, we already observed that the set of geometric points of~$\Theta(C)$ is finite and coincides with the set of $2^g-1$ theta-hyperplanes of~$C$.  Moreover, the morphism~$\gamma^\vee$ is \'etale of degree $(g-1)!$, corresponding to the possible orderings of the $(g-1)$ intersection points between a theta-hyperplane and the curve~$C$.

We are left to show that the multiplicity of each geometric point of~$\Theta(C)$, or, equivalently, of~$C_{2,g-1}^\circ$, is~$2^{g-1}$.  This follows from Proposition~\ref{prop:tp}.
\end{proof}

In the following theorem, we denote by
\[
\pgp \colon \Sgpdue \to \Mgdue
\aand
\pgm \colon \Sgmdue \to \Mgdue
\]
the reduction modulo~$2$ of the projection maps $\Sgp \to \Mg$ and $\Sgm \to \Mg$.  This theorem is the numerical analogue, in characteristic~$2$, of the tropical result \cite{jl}*{Theorem~1.1}.

\begin{Thm}\label{thm:pgr}
Let $(C,\theta) \in \Sgdue$ be a point corresponding to a smooth, projective, ordinary curve~$C$ of genus~$g$ over a field~$k$ of characteristic~$2$, with a theta-characteristic~$\theta$.  If $(C,\theta)$ is in~$\Sgpdue$, then the \pgpl~$\lambda^+_{(C,\theta)}$ at the point $(C,\theta)$ is
\[
\lambda^+_{(C,\theta)} = \left\{
\begin{array}{l@{\quad {\textrm{if }}}l}
2^g, & ~\theta {\textrm{ is the canonical theta-characteristic of }} C, \\[3pt]
2^{g-1}, & ~\theta {\textrm{ is not the canonical theta-characteristic of }} C.
\end{array}
\right.
\]
If $(C,\theta)$ is in~$\Sgmdue$, then the {\pgml}~$\lambda^-_{(C,\theta)}$ at the point $(C,\theta)$ is
\[
\lambda^-_{(C,\theta)} = 2^{g-1}.
\]
\end{Thm}

\begin{proof}
Let $\mathcal{S}_C$ be the fiber of $\pg \colon \Sgdue \to \Mgdue$ over the $k$-valued point corresponding to the curve~$C$.  Thus,~$\mathcal{S}_C$ is a scheme whose geometric points correspond to theta-characteristics of~$C$.  Observe that~$\mathcal{S}_C$ is a torsor under the $2$-torsion subgroup scheme $\Jac_{C}[2]$ of $\Jac_{C}$.  The group-scheme $\Jac_{C}[2]$ has length~$2^{2g}$ and, since~$C$ is ordinary, it has~$2^g$ distinct geometric points.  By homogeneity, each geometric point of $\Jac_{C}[2]$ has multiplicity~$2^g$.  Therefore, the $\Jac_{C}[2]$-torsor $\mathcal{S}_C$ also consists of $2^g$ geometric points, each of multiplicity~$2^g$.

From the description of~$\mathcal{S}_C$, we deduce the equality
\begin{equation} \label{e:sotc}
\lambda^+_{(C,\theta)} + \lambda^-_{(C,\theta)} = 2^g,
\end{equation}
with the convention that, for $\varepsilon \in \{+,-\}$, if $(C,\theta) \in \Sgdue$ is not in~$\due{\mathcal{S}}_g^\varepsilon$, then $\lambda^\varepsilon_{(C,\theta)} = 0$.  (A posteriori, the only new notation that this convention introduces is $\lambda^-_{(C,\theta)}$, where~$\theta$ is the canonical theta-characteristic of~$C$.)

We deduce that, to prove the result, it suffices to prove the identities
\begin{equation}\label{e:lambdameno}
\lambda^-_{(C,\theta)} = \left\{
\begin{array}{l@{\quad {\textrm{if }}}l}
0, & ~\theta {\textrm{ is the canonical theta-characteristic of }} C, \\[3pt]
2^{g-1}, & ~\theta {\textrm{ is not the canonical theta-characteristic of }} C.
\end{array}
\right.
\end{equation}

By~\cite{sv}*{p.~60}, if~$\theta$ is the canonical theta-characteristic of~$C$, then~$\theta$ is not effective.  It follows that the pair $(C,\theta)$ does not belong to the stack $\Sgmdue$.  Thus, $\lambda^-_{(C,\theta)}$ vanishes.  This proves the result in the case in which~$\theta$ is the canonical theta-characteristic.

Suppose now that~$\theta$ is not the canonical theta-characteristic.  Recall that the supports of~$\Sgmdue$ and~$\Snc$ coincide.  By the irreducibility of $\Snc$, Theorem~\ref{t:eke}, the function $\lambda^-_{(C,\theta)}$ is independent of the pair $(C,\theta)$, over a non-empty open subset of $\Mgdue$.  Denote by~$a$ the value of $\lambda^-_{(C,\theta)}$ on this open set.

We now show that~$a$ equals~$2^{g-1}$.  If~$C$ is an ordinary curve, then~$C$ has $2^g-1$ non-canonical theta-characteristics.  It follows that, for a general curve, the equality
\[
a (2^g-1) = \deg \pi_g^- = \deg \pi_g|_{\Sgm}
\]
holds. Since a smooth curve over a field of characteristic~$0$ has $2^{g-1} (2^g-1)$ odd theta-characteristics, we deduce the equality
\[
\deg \pi_g|_{\Sgm} = 2^{g-1} (2^g-1).
\]
We obtain that~$a$ equals $2^{g-1}$, as required.

So far, we proved the theorem for a general curve.  We now argue that the result holds for all ordinary curves.  Suppose, now, that~$C$ is an ordinary curve.  As before, we reduce to showing the identities in~\eqref{e:lambdameno}.  Let $\mathcal{S}_C^- \subset (\pg^-)^{-1}(C)$ be the fiber of~$\pg^-$ on the curve~$C$.  By the flatness of~$\pg^-$, the formula
\[
2^{g-1} (2^g-1) = \sum_{(C,\theta) \in \mathcal{S}_C^-} \lambda^-_{(C,\theta)}
\]
holds.  By upper-semicontinuity of lengths of fibers, each term $\lambda^-_{(C,\theta)}$ in the sum above satisfies the inequality $\lambda^-_{(C,\theta)} \geq 2^{g-1}$.  To conclude, it is sufficient to show that~$\mathcal{S}_C^-$ consists of $2^g-1$ distinct geometric points.

The fibers of~$\pg$ over the ordinary locus consist of~$2^g$ distinct geometric points, of which $2^g-1$ correspond to pairs $(D,\eta)$ with~$\eta$ a non-canonical theta-characteristic on~$D$.  Still on the ordinary locus, the support of each fiber of the morphism~$\pg^-$ is contained in the $2^g-1$ distinct geometric points of the fibers of~$\pg$ corresponding to non-canonical theta-characteristics.  We deduce that the number of distinct geometric points in the fibers of~$\pg^-$ cannot decrease on the ordinary locus, since the morphism~$\pg^-$ is proper.  We conclude that~$\mathcal{S}_C^-$ contains $2^g-1$ distinct geometric points.  The result follows.
\end{proof}

\section{Final remarks}

The method developed in this paper is potentially applicable to further situations.  We give a few examples of problems where the same strategy could go through.  We leave it as a challenge to check whether or not it works out.

\subsection{Steiner's conic problem}
This is the question of how many smooth conics are tangent to five general conics in the plane.  Several people worked on this problem (Steiner, de Jonqui\`eres, Chasles, Fulton, MacPherson).  The number over fields of characteristic different from~$2$ is~$3264$ (see~\cite{eiha} for more information on the problem).  In~\cite{vcon}, Vainsencher worked out the number over fields of characteristic~$2$ and obtained~$51$ as the answer (see~\cite{vcon}*{Section~9}).  We observe that $3264=2 \cdot 2^5 \cdot 51$.  We can imagine using Proposition~\ref{prop:tp} to show that the smooth conics tangent to five general conics over a field of characteristic~$2$ give rise to a finite scheme whose points have length is~$2^5$.  Combining this with Vainsencher's count, gives $2^5 \cdot 51$ conics counted with multiplicity.  This is half of what we expect from the results in characteristic different from~$2$.  The missing $2^5 \cdot 51$ solutions, could arise from counting singular conics appropriately.

\subsection{The formula of de Jonqui\`eres} \label{su:dej}
Choose non-negative integers $g,v$ and a $t$-uple of positive integers $\underline{m} = (m_1 , \ldots , m_t)$.  The de Jonqui\`eres number (\cite{dej}*{Theor\`eme~I,~3a}) is
\[
J(g,v , \underline{m}) = \left( \prod m_i \right) \sum _{h=0}^t \binom{t+v-h}{v} \, \binom{g}{h} \, (t-h)! \, h! \, \sigma_h (m_1-1 , \ldots , m_t-1),
\]
where $\sigma_0,\ldots , \sigma_t$ are the elementary symmetric functions in~$t$ variables.  Fix a smooth projective curve $C \subset \mathbb{P}^n$ of genus~$g$ and degree $d = v+ \sum m_i$.  Under appropriate conditions, the integer $J(g,v , \underline{m})$ is the expected number of hyperplanes that have contact multiplicities at least $m_1 , \ldots , m_t$ at~$t$ ordered points of~$C$.  A special Jonqui\`eres numbers yield the Pl\"ucker formulas (\cite{plu}) for the number of inflection and bitangent lines of general plane curves.  See~\cite{vdj}*{Section~4.4, p.~407} for a modern treatment of the de Jonqui\`eres formula.  In the special case of bitangents to plane quartics, we find
\[
J(3,0,(2,2)) = 56.
\]
This is twice the number of bitangent lines, since we are ordering the points of tangency between the bitangent and the curve. 
Writing
\[
(t-h)! \, h! \, \sigma_h (m_1-1 , \ldots , m_t-1) = \sum _{\tau \in \mathfrak{S}_t} (m_{\tau(1)}-1) \cdots (m_{\tau(t)}-1),
\]
it is clear that $\#{\textrm{Stab}}_{\mathfrak{S}_t}(\underline{m})$ divides the integer $(t-h)! \, h! \, \sigma_h (m_1-1 , \ldots , m_t-1)$.  This shows that the number
\[
\frac{J(g,v , \underline{m})}{\#{\textrm{Stab}}_{\mathfrak{S}_t}(\underline{m})}
\]
is an integer divisible by $\prod m_i$.  This number counts hyperplanes as above, without an ordering of the points.  In the case in which the integers $m_1,\ldots,m_t$ are all equal to~$2$, this is perfectly in line with the formula in Proposition~\ref{prop:tp}.  Moreover, Lemma~\ref{l:pl3} gives similar divisibility properties for the de Jonqui\`eres number, in the case of inflection lines of plane curves.  This suggests a possible approach to proving the divisibility of the de Jonqui\`eres number by the product $\prod m_i$ following the methods used in this paper.

\subsection{A log Gromov-Witten invariant with maximal contact multiplicities} \label{su:gw}

Let $L \subset \mathbb{P}^2$ be a line, let $Q \subset \mathbb{P}^2$ be a conic not tangent to~$L$ and let $x \in \mathbb{P}^2$ be a general point.  For a fixed positive integer~$d$, denote by~$N_d$ the number of nodal rational curves of degree~$d$ containing~$x$ and having a single point of intersection, of multiplicity~$d$, with~$L$ and a single point of intersection, of multiplicity~$2d$, with~$Q$.  The number~$N_d$ is finite and it has been computed by~\cite{bbg}:
\[
N_d = \binom{2d}{d}.
\]
This result is part of a more general project relating log Gromov-Witten invariants with local Gromov-Witten invariants (see the log-local principle, \cite{grg}*{Conjecture~1.4}).  This computation also follows from the general results in~\cite{nr}.

Suppose that the integer~$d$ is a prime number~$p$.  Analogously to our Lemmas~\ref{l:pl3} and~\ref{l:lu2}, we would expect that each nodal rational curve with the required maximal contact order with $L \cup Q$ would give a contribution divisible by~$p^2$ to the count: one factor of~$p$ for each contact point.  Moreover, the two tangent lines $L_1,L_2$ to~$Q$ through~$x$ give rise to two degenerate ``solutions''  $pL_1$ and $pL_2$.  Thus, we would expect the number~$N_p$ to satisfy the congruence
\[
N_p = \binom{2p}{p} \equiv 2 \pmod{p^2}.
\]
Indeed, this congruence holds for all primes~$p$.  The congruence actually holds modulo~$p^3$, as long as~$p$ is different from~$2$ or~$3$.  The further divisibility could be explained by the higher order of contact~$2p$ with the curve~$Q$.

\subsection{Several primes}
In some problems, the natural {\emph{contact}} is composite, instead of prime.  For instance, this happens when considering the $m$-torsion subgroup of an abelian variety.  In this specific case, the information coming separately from each prime power-factor of~$m$ could be combined (e.g.\ via the Chinese Remainder Theorem).  In other cases, this is not so clear.  For instance, the formula of de Jonqui\`eres and the Gromov-Witten example (Subsections~\ref{su:dej} and~\ref{su:gw}) naturally have contact multiplicities that are not necessarily prime numbers.  In these cases, it is not clear whether the information at individual primes is sufficient to deduce the answer for every integer.  This seems a challenging and interesting question.

\end{document}